\documentclass[10pt]{amsart}
\usepackage{epsfig,url}
\usepackage{mathdots}

\newtheorem{theorem}{Theorem}[section]
\newtheorem{lemma}[theorem]{Lemma}

\newtheorem{corollary}[theorem]{Corollary}

\theoremstyle{definition}

\theoremstyle{remark}

\begin{document}

\title[An entry in GR and KdV]{Collisionless shock region of the KdV equation and an entry in 
Gradshteyn and Ryzhik}

\medskip

\author[T. Amdeberhan et al]{Tewodros Amdeberhan}
\address{Department of Mathematics,
Tulane University, New Orleans, LA 70118}
\email{tamdeber@tulane.edu}

\author[]{Victor H. Moll}
\address{Department of Mathematics,
Tulane University, New Orleans, LA 70118}
\email{vhm@tulane.edu}

\author[]{John Lopez Santander}
\address{Department of Mathematics,
Tulane University, New Orleans, LA 70118}
\email{jlopez12@tulane.edu}

\author[]{Ken McLaughlin}
\address{Department of Mathematics,
Tulane University, New Orleans, LA 70118}
\email{kmclaughlin@tulane.edu}

\author[]{Christoph Koutschan}
\address{RICAM, Austrian Academy of Sciences}
\email{christoph@koutschan.de}


\date{\today}


\maketitle

\newcommand{\ba}{\begin{eqnarray}}
\newcommand{\ea}{\end{eqnarray}}
\newcommand{\ift}{\int_{0}^{\infty}}
\newcommand{\nn}{\nonumber}
\newcommand{\no}{\noindent}
\newcommand{\lf}{\left\lfloor}
\newcommand{\rf}{\right\rfloor}
\newcommand{\realpart}{\mathop{\rm Re}\nolimits}
\newcommand{\imagpart}{\mathop{\rm Im}\nolimits}
\newcommand{\K}{\mathbf{K}}
\newcommand{\J}{\mathbf{J}}
\newcommand{\A}{\mathbf{A}}

\newcommand{\op}[1]{\ensuremath{\operatorname{#1}}}
\newcommand{\pFq}[5]{\ensuremath{{}_{#1}F_{#2} \left( \genfrac{}{}{0pt}{}{#3}
{#4} \bigg| {#5} \right)}}

\newcommand{\pFqcomma}{\mskip\pFqmuskip}

\newtheorem{Definition}{\bf Definition}[section]
\newtheorem{Thm}[Definition]{\bf Theorem}
\newtheorem{Example}[Definition]{\bf Example}
\newtheorem{Lem}[Definition]{\bf Lemma}
\newtheorem{Cor}[Definition]{\bf Corollary}
\newtheorem{Prop}[Definition]{\bf Proposition}
\numberwithin{equation}{section}

\section*{}
\begin{center}
\textit{To the memory of  Hermann Flaschka}
\end{center}

\begin{abstract}
The long-time  behavior of solutions to the initial value problem for the Korteweg-de Vries equation on the whole
 line, with general initial conditions has been described uniformly using five different asymptotic forms. Four of 
 these asymptotic forms were expected:  the \textit{quiescent behavior} (for $|x|$ very large), a \textit{soliton region} (in which the solution behaves as a collection of isolated solitary waves), a \textit{self-similar region}
 (in which the solution is described via a Painlev\'{e} transcendent), and a \textit{similarity region} (where the solution behaves as a simple trigonometric function of the quantities $t$ and $x/t$).   A fifth asymptotic form, 
lying between the self-similar (Painlev\'{e}) and the similarity one, has been described in terms of classical 
elliptic functions. An integral of elliptic type, giving an explicit representation of the phase, has appeared in this 
context. The same integral has appeared in the 
table of integrals by Gradshteyn and Ryzhik.  Our goal here is to confirm the validity of this entry. 
\end{abstract}

\section{Introduction}
\label{sec-introduction}

The table of integrals created by Gradshteyn and Ryzhik \cite{gradshteyn-2015a} contains a large variety of entries where the 
answers are expressed in terms of the complete elliptic integral
\begin{equation*}
\K(k) = \int_{0}^{1} \frac{dx}{\sqrt{(1-x^{2})(1-k^{2}x^{2})}}. 
\end{equation*}
\noindent
The goal of this paper is to present  a proof of entry $4.242.4$ in \cite{gradshteyn-2015a}:
\begin{eqnarray}
\label{entry-4.242.4}
I(a,b) & := &  \int_{0}^{b} \frac{\ln x \, dx}{\sqrt{(a^{2}-x^{2})(b^{2}-x^{2})}}  \\ 
& = & 
\frac{1}{2a} \left[ \K \left( \frac{b}{a} \right) \ln(ab) - \frac{\pi}{2} \K \left( \frac{\sqrt{a^{2} - b^{2}}}{a} \right) \right]  \nonumber 
\end{eqnarray}
\noindent 
with $0 < b < a$. Note that the right-hand side equals 
$\frac{1}{2a} \left[ \K(k) \ln(ab) - \frac{\pi}{2} \K(k') \right]$, with modulus $k = b/a$ and complementary modulus $k' = \sqrt{1- k^{2}}$.  

Some integrals of this general type have been considered in \cite{boettner-2011a} as part of a series of 
papers dedicated to establishing all entries in \cite{gradshteyn-2015a} starting with \cite{moll-2007a} and 
currently at \cite{gr-32}.  Entry $4.242.1$ 
\begin{equation*}
\int_{0}^{\infty} \frac{\ln x \, dx}{\sqrt{(a^{2}+x^{2})(b^{2}+x^{2})}} = \frac{1}{2a} \K \left( \frac{\sqrt{a^{2}-b^{2}}}{a} \right) \ln(ab)
\end{equation*}
\noindent
has been proved in \cite{boettner-2011a}. The crucial point in the proof comes down to the identity 
\begin{equation}
\label{stefan-1}
\sum_{\ell=0}^{j-1} \frac{a_{\ell}}{j- \ell} = 4 a_{j} \sum_{\ell=0}^{j-1} \frac{1}{2 \ell + 1}
\,\,\, \textnormal{ where } \,\,\, 
a_{\ell} = \frac{ \left( \tfrac{1}{2} \right)_{\ell}^{2}}{\ell!^{2}}.
\end{equation}
\noindent
Two proofs of \eqref{stefan-1} were presented: one involves the manipulation of a balanced ${_{4}F_{3}}$ hypergeometric series and the 
other is an automatic proof based on the techniques developed in \cite{petkovsek-1996a}. It is our aim 
 to proceed in a similar manner to verify \eqref{entry-4.242.4}. 

Section \ref{sec-analytic-1} presents an expression for $I(a,b)$ in terms of the power series 
\begin{equation}
F(x) = \sum_{n=0}^{\infty} (-1)^{n} \binom{2n}{n}^{2} H_{n} x^{n}, 
\label{F-def} 
\end{equation}
\noindent
where $H_{n} = 1 + \tfrac{1}{2} + \tfrac{1}{3} + \cdots + \tfrac{1}{n}$ is the harmonic number.  The series
 in \eqref{F-def} converges for 
$|x| < \tfrac{1}{16}$, it is not a hypergeometric function and it makes its appearance in Lemma \ref{lemma-Jc1}.
A remarkable fourth order differential equation for $F$, presented in Section \ref{sec-automatic}, is 
then used to supply an automated proof of \eqref{entry-4.242.4}.

\section{Background on the  integral}


The Korteweg-de Vries equation (KdV)
\begin{eqnarray*}
u_{t} - 6 u u_{x} + u_{xxx} = 0 
\end{eqnarray*}
originated in the $19^{th}$ century  as a  description of the evolution of long waves in shallow water such as a canal \cite{boussinesq-1871a,korteweg-1895a}.  Much later (in the late 1960s) an incredible connection was discovered between this equation and the scattering and inverse scattering theory of the 1-dimensional Schr\"{o}dinger operator \cite{gardner-1967b}
\begin{eqnarray}
\label{eq:schro}
 L=-\frac{d^{2}}{dx^{2}} + u  \ .
\end{eqnarray}
For potentials  $u(x)$ that are rapidly decaying as $|x| \to \infty$, the scattering data for $L$ consists of
$(i)$  a reflection coefficient $r(z)$ describing the energy that is reflected back when an incoming (quantum) wave with velocity $z$ interacts with the medium represented by $u$, and $(ii)$ 
 a finite number of eigenvalues and associated normalization constants (each eigenvalue  corresponds to a bound state for the operator $L$).  

The amazing discovery in \cite{gardner-1967b} is that if the potential $u$  in \eqref{eq:schro} evolves according to the KdV equation, so that now $u$ depends on time $u = u(x,t)$ and as does the 
operator $L = L(t)$, this nonlinear evolution has two remarkable properties:
\begin{itemize}
\item The eigenvalues of the operator $L(t)$ are \textit{constant in time} and 
 the associated normalization constants  evolve in a simple manner. 
\item The reflection coefficient $r=r(z,t)$ evolves explicitly in $t$:
\begin{equation*}
r(z,t) = r(z,0) e^{ 8 i z^{3} t }
\end{equation*}
where $r(z,0) = r_{0}(z)$ is the reflection coefficient corresponding to the initial potential $u(x,0) = u_{0}(x)$.
\end{itemize}
\noindent
These remarkable properties followed hard upon the heels of the earlier discovery of Zabusky and Kruskal \cite{zabusky-1965a}. 
Previously these authors had made the observation that the KdV 
equation possesses not only the classical solitary wave solution, but also other special solutions that spread out 
at long times like a collection of separated individual solitary waves (both at large negative as well as large
 positive times), but interact with each other at intermediate times that seem to defy their interpretation as
  individual solitary waves.  Shortly after  Gardner, Green, Kruskal, and Miura made the connection to the 
  Schr\"{o}dinger operator, Lax \cite{lax-1968a} introduced the framework that bears his name. The  pair of operators discovered in 
\cite{gardner-1967b} represented the first example of a \textit{Lax pair,} which brought forth the so-called inverse scattering method for the analysis of certain special nonlinear partial differential equations.

 
 From these origins there emerged numerous collections of nonlinear partial differential equations and associated Lax pairs of operators.  Primary amongst these was the discovery of Zakharov and Shabat \cite{ zakharov-1973b} that the nonlinear Schr\"{o}dinger equation 
  falls into the framework, and the discovery independently by Flaschka \cite{ flaschka-1974b, flaschka-1974a} and Manakov \cite{ manakov-1974b, manakov-1974c} that the Toda lattice does, too.   In each case, the scattering data for one of the operators evolves simply in the time variable. The inverse scattering theory was then exploited, or developed, to yield a solution procedure for these new integrable nonlinear partial differential equations. As it turns out, the solution procedure is extremely powerful. Using it, researchers discovered remarkable phenomena in the behavior of these nonlinear equations which turns out to be ubiquitous even outside the class of integrable equations.

A basic example of this type of discovery is the complete understanding of the long-time behavior of solutions to these equations \textit{from general initial conditions}.  This was first established for 
the nonlinear Schr\"{o}dinger equation \cite{zakharov-1976a}.  It was presumed for some time that the calculations for the KdV equation would be entirely similar, but a curious technical obstacle to the direct application of the method reared its head, and led to a new asymptotic phenomenon discovered in the behavior of general solutions of the KdV equation, the so-called \textit{collisionless shock region} for the KdV equation
 \cite{ablowitz-1977b}.

For generic initial conditions, the behavior of the solution of the KdV equation is described uniformly 
using five  different asymptotic forms each occurring in a different spatial region.  Four of these asymptotic forms were expected:  the \textit{quiescent behavior} (for $|x|$ very large), a \textit{soliton region} (in which the solution behaves as a collection of isolated solitary waves), a \textit{self-similar region}
 (in which the solution is described via a Painlev\'{e} transcendent), and a \textit{similarity region} (where the solution behaves as a simple trigonometric function of the quantities $t$ and $x/t$).  

The fifth region, it turns out, lies in between the self-similar (Painlev\'{e}) and the similarity one.  It emerges
 because of the surprisingly benign fact that for generic potentials (i.e.,
  generic initial data for the KdV equation), the reflection coefficient takes on a specific extreme value at $z=0$:  $r(0)=-1$.  The reflection coefficient satisfies $|r(z)| \le 1$, and for all values of $z$ \textit{other than $0$}, this inequality is \textit{strict}, but not so at $z=0$.  In the calculations, $|x|, t \to \infty$ at different rates depending on which asymptotic region one is studying, and 
the quantity $\log{\left( 1 - \left|
r\left( \frac{x}{12t}\right)
\right|^{2} \right) }$  is a key ingredient in the asymptotic description.  The fact that this quantity can diverge, because $r(0) = -1$, is the source of the fifth region.

In this new region, the behavior of the solution is described in terms of a Jacobi cnoidal function 
\cite{whittaker-2020a}.  In explicit form, the asymptotic form of the solution in this region is
\begin{equation}
\label{eq:JacobiEll}
u \sim \left( \frac{ - 2 x }{3 t} \right) \left( a(\alpha) + b(\alpha) \text{cn}^{2}\left( 2 \K(\nu) \theta + \theta_{0} ; \nu(Z) \right) \right) \ ,
\end{equation}
\noindent
where $\text{cn}(\cdot; \nu)$ represents the Jacobi elliptic function \cite{whittaker-2020a}, the quantity 
$\alpha$ is determined to depend on a slow variable $Z$, $\theta$ is a fast variable, and $K(\alpha)$ is the complete elliptic integral of the first kind.  The quantity $\theta_{0}$ is a phase which was undetermined in the original work of Albowitz and Segur.

This work led to foundational questions regarding how to provide a rigorous proof of the asymptotic formulae in \cite{ablowitz-1977b} and, more generally, of  the long-time analysis of integrable nonlinear partial differential 
equations.  Such problems remain outside of the reach of any classical methods.  The Riemann-Hilbert 
machinery, developed by Deift and Zhou \cite{deift-1993c,deift-1992b,deift-1993a,deift-1994d} for integrable problems, was used by 
Deift, Venakides, and Zhou  \cite{deift-1994a} to analyze the collisionless shock region.

In the scattering and inverse scattering theory applied to the modified KdV equation in \cite{deift-1992b}, and to the KdV equation in \cite{deift-1994a}, the solution to the KdV equation was characterized through the solution of a vector-valued Riemann-Hilbert problem.  In order to carry out the long-time analysis of this Riemann-Hilbert problem, and extract the long-time behavior of the solution to the partial differential equation, the authors invented explicit transformations relating one Riemann-Hilbert problem to another, each transformation in turn simplifying the nature of the subsequent one, until arriving at a final one for which an existence and uniqueness theorem could be established.  Unraveling
the sequence of transformations,  precise analytical descriptions of the behavior of the solution to the 
 partial differential equation can be extracted. 

In each different asymptotic region, the sequence of transformations is different, and tailored to extract from the original Riemann-Hilbert problem the dominant contributing elements to the eventual asymptotic form of
 the solution.  In fact, for the first four regions, the sequence of transformations showed that the dominant 
 contribution comes from a finite number (usually one or two) of  isolated points called stationary phase points
  in the spectral plane.

However, the Riemann-Hilbert analysis for the new collisionless shock region presented a leap in complexity.  
Indeed, the dominant contribution arose from an evolving (finite) collection of intervals in the spectral plane.  Specifically, for $(x,t)$ in the collisionless shock region, four real endpoints emerged: 
$\pm a(x,t)$, $\pm b(x,t)$, with $0 < a(x,t) < b(x,t)< \sqrt{2}$, and $a^{2} + b^{2}=2$. These define 
 intervals $(-b,-a)$ and $(a,b)$.  The authors used the Riemann surface $\mathbb{X}$ associated to the function
$f(z) = \left( z + b \right)^{1/2}
\left( z + a \right)^{1/2}
\left(z - a \right)^{1/2}
\left(z - b \right)^{1/2},$
as a fundamental ingredient in their analysis.  The  genus of $\mathbb{X}$ is $1$ and 
the Jacobi elliptic function appearing in \eqref{eq:JacobiEll} is constructed using the Abel map and the periods associated to this surface.

The integral \eqref{entry-4.242.4} appears  in \cite{deift-1994a}, when the authors established 
an explicit representation of the phase $\theta_{0}$ (for arbitrary initial conditions): 
\begin{eqnarray*}
\theta_0 & = &   \K(\alpha)-\int_1^{\sqrt{b / a}}\left(\left(w^2-1\right)\left(1-(a / b)^2 w^2\right)\right)^{-1 / 2} d w \\
& &  -\frac{1}{2 \pi b } \int_{-a}^{a} \frac{\log \left(2 \gamma w^{2} \right)}{\left(\left(w^2-a^{2}\right)\left(w^2-b^{2} \right)\right)^{1 / 2}} d w. \nonumber
\end{eqnarray*}
The explicit determination of the asymptotic form \eqref{eq:JacobiEll}, including the phase $\theta_{0}$ as well as the size of the error term, highlights the reach of integrability.  

But the occurrence of the integral \eqref{entry-4.242.4} is more than just as a part of the answer.  In order to 
understand the overlap between the collisionless shock region and its two neighboring regions, one must study
 the behavior of the Riemann surface $\mathbb{X}$, the Jacobi elliptic function, and all internal parameters, in
  two 
singular limits: in one of them the branch point $a$ converges to $0$, and in the other one the branch 
points $a$ and $b$ merge together.  Having an explicit representation of this integral in terms of well-known special functions assists greatly in matching  the asymptotic form of the KdV solution in the transition regions.  
There lies the value of \eqref{entry-4.242.4}.

\section{An analytic representation of the integral $I(a,b)$}
\label{sec-analytic-1}

This section presents an analytic expression for the integral $I(a,b)$ in \eqref{entry-4.242.4} in terms of  the function $F$ from \eqref{F-def}.
The change of variables $t = x/b$  yields 
\begin{equation*}
I(a,b) = \frac{\ln b}{a} \int_{0}^{1} \frac{dt}{\sqrt{(1-t^{2})(1-c^{2}t^{2})}} + \frac{1}{a} \int_{0}^{1} \frac{\ln t \, dt}{\sqrt{(1-t^{2})(1-c^{2}t^{2})}}
\end{equation*}
\noindent
with $c = b/a$. Observe that $0< c< 1$. The first integral is $\K(c)$ and so
\begin{equation}
\label{transf-1}
I(a,b)  = \frac{\ln b}{a} \K(c) + \frac{1}{a} \J(c),
\end{equation}
\noindent
where
\begin{equation}
\J(c) = \int_{0}^{1} \frac{\ln x \, dx}{\sqrt{(1-x^{2})(1-c^{2}x^{2})}}.
\label{J-def-0}
\end{equation}

\begin{lemma}
\label{lemma-Jc1}
The integral $\J(c)$ in \eqref{J-def-0} is given by 
\begin{equation*}
\J(c) = -  \ln 2 \, \K(c) - \frac{\pi}{4 \sqrt{1-c^{2}}} F \left( \frac{c^{2}}{16(1-c^{2})} \right).
\end{equation*}
\end{lemma}
\begin{proof}

The change of variables $s = t^{2}$ yields 
\begin{equation*}
\J(c) = \frac{1}{4} \int_{0}^{1} \frac{\ln s \, ds}{\sqrt{s (1-s)(1- c^{2}s)}}.
\end{equation*}
\noindent
To evaluate $\J(c)$ consider the function 
\begin{equation*}
\A(c,r) = \int_{0}^{1} s^{r} s^{-1/2} (1-c^{2} s)^{-1/2} (1-s)^{-1/2} \, ds
\end{equation*}
\noindent
and observe that
\begin{equation*}
\J(c) = \frac{1}{4}\frac{d}{dr} \A(c,r)\Big{|}_{r=0}.
\end{equation*}
\noindent
The integral representation (\cite[$9.111$]{gradshteyn-2015a}) of the hypergeometric function
\begin{equation*}
\pFq21{\alpha \,\, \beta}{\gamma}{z} = \frac{1}{B(\beta, \gamma- \beta)} \int_{0}^{1} t^{\beta-1} (1-t)^{\gamma - \beta - 1} (1 - tz)^{-\alpha} \, dt 
\end{equation*}
\noindent
implies
\begin{equation*}
\A(c,r) = B \left( r + \frac{1}{2}, \frac{1}{2} \right) \pFq21{\frac{1}{2} \,\, \, \, r + \tfrac{1}{2}}{r+1}{\, c^{2}}.
\end{equation*}
\noindent
The relation (see \cite[$9.131.1$]{gradshteyn-2015a})
\begin{equation*}
\pFq21{\alpha \,\, \beta}{\gamma}{z} =  (1-z)^{-\alpha} \pFq21{\alpha \,\,\,\,  \gamma - \beta}{\gamma}{\frac{z}{z-1}}
\end{equation*}
\noindent 
yields 
\begin{equation*}
\A(c,r) = (1-c^{2})^{-1/2} B \left( r + \frac{1}{2}, \frac{1}{2} \right) \pFq21{ \tfrac{1}{2} \,\,\,\,  \tfrac{1}{2}}{r+1}{\frac{c^{2}}{c^{2}-1}}.
\end{equation*}
\noindent
Write  $\A(c,r) = (1-c^{2})^{-1/2} A_{1}(r) C_{1}(c,r)$ where
\begin{equation*}
A_{1}(r)  = B \left( r + \tfrac{1}{2}, \tfrac{1}{2} \right) \quad \textnormal{and} \quad  C_{1}(c,r) = \pFq21{ \tfrac{1}{2} \,\,\,\,  \tfrac{1}{2}}{r+1}{\frac{c^{2}}{c^{2}-1}},
\end{equation*}
\noindent
so that 
\begin{equation*}
\J(c) = \frac{1}{4} \frac{d}{dr} \A(c,r)\Big{|}_{r=0} =  \frac{1}{4 \sqrt{1-c^{2}}} \left[ A_{1}(0) C_{1}'(c,0) + A_{1}'(0) C_{1}(c,0) \right],
\end{equation*}
\noindent
where $C_{1}'$ is the derivative with respect to $r$. Each of these four terms are evaluated individually. 

\smallskip 

\noindent
\texttt{First term}: $A_{1}(0)$. The beta function is given by 
\begin{equation*}
B(x,y) = \frac{\Gamma(x) \Gamma(y)}{\Gamma(x+y)},
\end{equation*}
\noindent
so that 
\begin{equation*}
A_{1}(0) = \frac{\Gamma( r + \tfrac{1}{2}) \Gamma( \tfrac{1}{2})}{\Gamma(r+1)}\Big{|}_{r=0} = \pi. 
\end{equation*}

\smallskip 

\noindent
\texttt{Second term}: $C_{1}(c,0)$. The value is given by 
\begin{equation*}
C_{1}(c,0)  =   \pFq21{ \tfrac{1}{2} \,\,\,\, \,\,\, \tfrac{1}{2} }{1} {\frac{c^{2}}{c^{2}-1}}.
\end{equation*}
\noindent
On the other hand, this hypergeometric value corresponds to the complete elliptic integral (see \cite[8.113.1]{gradshteyn-2015a}) 
\begin{equation*}
\K(k) = \frac{\pi}{2}  \pFq21{ \tfrac{1}{2} \,\,\,\, \,\,\, \tfrac{1}{2} }{1} {k^{2}}
\end{equation*}
\noindent
so that  
\begin{equation*}
C_{1}(c,0) = \frac{2}{\pi} \K \left( \sqrt{ \frac{c^{2}}{c^{2}-1} } \right).
\end{equation*}

Observe that $ c^{2} = b^{2}/a^{2} < 1$, therefore the argument of $\K$ above is purely imaginary.  The transformation rule (of the 
imaginary modulus) reads
\begin{equation}
\label{K-imag1}
\K(it) = \frac{1}{\sqrt{1+t^{2}}} \K \left( \frac{t}{\sqrt{1+t^{2}}} \right).
\end{equation}
(see \cite[Page 82]{mckean-1997a}), which gives
\begin{equation*}
C_{1}(c,0) =  \frac{2}{\pi} \sqrt{1-c^{2}} \, \K(c).
\end{equation*}

\smallskip

\noindent
\texttt{Third term}: $A_{1}'(0)$.  The function $A_{1}(r)$ is given by 
\begin{equation*}
A_{1}(r) = B(r + \tfrac{1}{2}, \tfrac{1}{2}) = \Gamma( \tfrac{1}{2}) \frac{\Gamma( r + \tfrac{1}{2} )}{\Gamma(r+1)}.
\end{equation*}
\noindent
Differentiation at $r=0$ yields  
\begin{equation*}
A_{1}'(0) = \frac{\Gamma(\tfrac{1}{2})}{\Gamma^{2}(1) } \left[ \Gamma'(\tfrac{1}{2}) \Gamma(1) - \Gamma'(1) \Gamma(\tfrac{1}{2} ) \right]. 
\end{equation*}
\noindent
The digamma function $\psi$, defined by $ \Gamma'(x) = \psi(x) \Gamma(x)$, shows that
\begin{equation*}
A_{1}'(0) = \frac{\Gamma(\tfrac{1}{2} )}{\Gamma^{2}(1)} 
\left[ \Gamma(\tfrac{1}{2} ) \psi( \tfrac{1}{2} ) \Gamma(1) -  \Gamma(1) \psi(1)  \Gamma( \tfrac{1}{2}) \right],
\end{equation*}
\noindent
and  the special values 
$\Gamma(1) = 1, \, \Gamma( \tfrac{1}{2} ) = \sqrt{\pi}, \, \psi(1) = - \gamma, \, \psi(\tfrac{1}{2} ) = - \gamma - 2 \ln 2$
appearing in \cite[$8.338.1, \, 8.338.2, \, 8.366.1$ and $8.366.2$]{gradshteyn-2015a} generate
\begin{equation*}
A_{1}'(0) = - 2 \pi \ln 2.
\end{equation*}

\smallskip

\noindent
\texttt{Fourth term}: $C_{1}'(c,0)$.  Start with
\begin{equation*}
C_{1}(c,r)  =   \pFq21{ \tfrac{1}{2} \,\,\,\, \,\,\, \tfrac{1}{2} }{r+1} {\frac{c^{2}}{c^{2}-1}}
\end{equation*}
\noindent
and since $0< c^{2} = b^{2}/a^{2} < 1$, it is convenient to introduce the parameter
\begin{equation*}
t = \frac{c^2}{1-c^{2}},
\end{equation*}
\noindent 
so that $t>0$ and 
\begin{equation*}
C_{1}(c,r)  =   \pFq21{ \tfrac{1}{2} \,\,\,\, \,\,\, \tfrac{1}{2} }{r+1} {-t} = \sum_{\ell = 0}^{\infty} (-1)^{\ell} \frac{ \left( \tfrac{1}{2} \right)_{\ell}^{2}}{(r+1)_{\ell}} 
\frac{t^{\ell}}{\ell!}.
\end{equation*}
\noindent
 The only term that needs to be differentiated is the Pochhammer symbol $(r+1)_{\ell}$.
To accomplish this, proceed in the manner
\begin{equation*}
\frac{d}{dr} \frac{1}{(r+1)_{\ell}} = - \frac{1}{(r+1)_{\ell}^{2}} \frac{d}{dr} (r+1)_{\ell},
\end{equation*}
\noindent
and then expanding $(r+1)_{\ell} = (r+1)(r+2) \cdots (r+\ell)$ leads to
\begin{equation*}
\frac{d}{dr} (r+1)_{\ell} = (r+1)_{\ell} \sum_{j=1}^{\ell} \frac{1}{r+j}.
\end{equation*}
\noindent
Evaluating at $r=0$ results in 
\begin{equation*}
\frac{d}{dr} \frac{1}{(r+1)_{\ell}}\Big{|}_{r=0}  = - \frac{1}{(r+1)_{\ell}} \sum_{j=1}^{\ell} \frac{1}{r+j}\Big{|}_{r=0} = 
 - \frac{1}{\ell!} H_{\ell}
\end{equation*}
\noindent
where $H_{\ell}$ is the harmonic number $H_{\ell} = 1 + \tfrac{1}{2} + \cdots + \tfrac{1}{\ell}$.  The \texttt{fourth term} becomes
\begin{equation*}
C_{1}'(c,0) = - \sum_{\ell = 0}^{\infty} (-1)^{\ell} \binom{2 \ell}{\ell}^{2} \left( \frac{c^{2}}{16(1-c^{2})} \right)^{\ell} H_{\ell},
\end{equation*}
\noindent
using the identity 
\begin{equation}
\binom{\tfrac{1}{2}}{n}  = \frac{(-1)^{n-1}}{2^{2n} (2n-1)} \binom{2n}{n}
\label{binom-form1}
\end{equation}
\noindent 
to simplify the result.   The proof is complete.
\end{proof}


\smallskip

The integral $I(a,b)$  is now expressed in terms of the function $F$ defined in \eqref{F-def}.

\begin{theorem}
\label{thm-rep1}
Preserving the notations from Lemma \ref{lemma-Jc1}, we have 
\begin{equation*}
I(a,b) = \frac{1}{a} \ln \left( \frac{b}{2} \right) \K \left( \frac{b}{a} \right) - 
\frac{\pi}{4 \sqrt{a^{2}-b^{2}}} F\left(\frac{b^{2}}{16(a^{2}-b^{2})} \right).
\end{equation*}
\end{theorem}

\section{An analytic proof of the main identity}
\label{sec-anal-main}

The previous section has given an expression for 
\begin{equation*}
I(a,b) =  \int_{0}^{b} \frac{\ln x \, dx}{\sqrt{(a^2-x^2)(b^2-x^2)}}
\end{equation*}
\noindent
 in terms of the function $F$ defined in \eqref{F-def}. 
This  section delivers  a direct  analytic proof of the main identity displayed in \eqref{entry-4.242.4}.

\begin{theorem}
The following identity holds:
\begin{equation}
\label{main-1}
I(a,b) = \frac{1}{2a} \left[ \K \left( \frac{b}{a} \right)  \ln(ab) - \frac{\pi}{2} \K \left( \frac{\sqrt{a^{2}-b^{2}}}{a}  \right)\right]. 
\end{equation}
\end{theorem}
%

\begin{proof}
The evaluation  \eqref{main-1} has been reduced, in \eqref{transf-1},  to
\begin{equation*}
I(a,b)  = \frac{\ln b}{a} \K(c) + \frac{1}{a} \J(c),
\end{equation*}
\noindent
with $c = b/a$ and where, after an elementary change of variables, 
\begin{equation}
\label{int-j1b}
\J(c) = \frac{1}{4} \int_{0}^{1} \frac{\ln s \, ds}{\sqrt{s(1-s)(1-c^{2}s)}}.
\end{equation}
\noindent
So, the proof  in this section amounts to a direct computation of $\J(c)$.

Start by  transforming  the interval of integration to a half-line of the new variable $x$, defined 
 via $s = 3/(3x+c^{2}+1)$.  Then 
\eqref{int-j1b} becomes 
\begin{eqnarray}
\label{int-j2}
\J(c) & = & - \frac{1}{2} \int_{\tfrac{1}{3}(2-c^{2})}^{\infty} \frac{\log( x + \tfrac{1}{3}(c^{2}+1)) \, dx}{
\sqrt{4 ( x + \tfrac{c^{2}+1}{3}) (x + \tfrac{c^{2}-2}{3})(x + \tfrac{1-2c^{2}}{3} )}} \\
& = & - \frac{1}{2} \int_{\tfrac{1}{3}(2-c^{2})}^{\infty} 
\frac{\log( x + \tfrac{1}{2}(c^{2}+1)) \, dx }{\sqrt{4 x^{3} - g_{2}x - g_{3}}}, \nonumber 
\end{eqnarray}
\noindent
with 
\begin{equation*}
g_{2} = \tfrac{4}{3}(c^{4}-c^{2}+1) \quad \textnormal{and} \quad g_{3} = \tfrac{4}{27} (1+c^{2})(1-2 c^{2})(2 - c^{2}).
\end{equation*}
\noindent
The discriminant of the cubic is $g_{2}^{3} - 27 g_{3}^{2} = 16c^4(c+1)^2(c-1)^2 \neq 0$ 
for $0 < c < 1$.  The roots of the cubic are real and are given by
\begin{equation*}
e_{1} = \tfrac{1}{3}(2 - c^{2})> e_{2} = \tfrac{1}{3}(2c^{2}-1) > e_{3} = - \tfrac{1}{3}(c^{2}+1).
\end{equation*}

Consider the  curve $\mathcal{E}$ defined by the equation
\begin{equation*}
y^{2} = 4 x^{3} - g_{2}x - g_{3}.
\end{equation*}
\noindent
It is well-known that (the projectivation of) $\mathcal{E}$ is a torus $\mathbb{C}/(\mathbb{Z} 2 \omega_{1} \oplus \mathbb{Z } 2\omega_{2})$, with an 
associated Weierstrass $\wp$-function satisfying 
\begin{equation*}
(\wp'(z))^{2} = 4 \wp(z)^{3} - g_{2} \wp(z) - g_{3} = 4(\wp(z) - e_{1})(\wp(z) - e_{2})(\wp(z) - e_{3}).
\end{equation*}
The parameters $\omega_{1}, \, \omega_{2}$ being the half-periods for $\wp(z)$, satisfy 
\begin{equation*}
\wp(\omega_{1}) = e_{1}, \quad \wp(\omega_{2}) = e_{3} \quad \textnormal{and} \quad \wp(\omega_{3}) = e_{2},
\end{equation*}
\noindent
with $\omega_{3} = \tfrac{1}{2}(\omega_{1}+\omega_{2})$. The periods can be written, using $P(x) = 4x^{3} - g_{2}x - g_{3}$, as 
\begin{equation*}
\frac{\omega_{1}}{2} = \int_{\infty}^{e_{1}} \frac{dt}{\sqrt{P(t)}} \quad \textnormal{and} \quad \frac{\omega_{2} }{2}= \int_{e_{1}}^{e_{2}} \frac{dt}{\sqrt{P(t)} }
\end{equation*}
\noindent
so that $\omega_{1}$ is real  and $\omega_{2}$ is purely imaginary.  See \cite{whittaker-2020a} for details. 

The integral in \eqref{int-j2} now becomes 
\begin{equation*}
\J(c) = - \frac{1}{2} \int_{e_{1}}^{\infty} \frac{\log( x - e_{3}) \, dx}{\sqrt{4(x-e_{1})(x-e_{2})(x-e_{3})}}. 
\end{equation*}
\noindent
Substitute $x = \wp(z)$ with $\wp'(z) = - \sqrt{4(\wp(z) - e_{1})(\wp(z)-e_{2})(\wp(z) - e_{3})}$ (the negative sign is due to 
the fact that  since $\wp$ is real and decreasing in the interval of integration). Since $\wp$ is an 
even function, it follows that 
\begin{equation}
\label{int-j3}
\J(c) =- \frac{1}{2} \int_{0}^{\omega_{1}} \log( \wp(z) - \wp(\omega_{2})) \, dz =
 - \frac{1}{4} \int_{-\omega_{1}}^{\omega_{1}} \log( \wp(z) - \wp(\omega_{2})) \, dz.
\end{equation}
\noindent

At this point, introduce the Weierstrass zeta-function $\zeta(z;g_{2},g_{3})$ (\cite[page 467]{whittaker-2020a}) defined by 
\begin{equation*}
\frac{d}{dz} \zeta(z;g_{2},g_{3}) = - \wp(z; g_{2},g_{3}) 
\end{equation*}
\noindent
and the normalization $\begin{displaystyle} \lim\limits_{z \rightarrow 0 } \left[ \zeta(z;g_{2},g_{3}) - 1/z \right] = 0. \end{displaystyle}$  The Weierstrass
 sigma-function (\cite[page 469]{whittaker-2020a})  is then  defined by 
\begin{equation*}
\frac{d}{dz} \log \sigma(z;g_{2},g_{3}) = \zeta(z;g_{2},g_{3})
\end{equation*}
\noindent
with the corresponding 
normalization  $\begin{displaystyle} \lim\limits_{z \rightarrow 0 } \frac{\sigma(z)}{z} = 1. \end{displaystyle}$ The parameters $g_{2}$ and $g_{3}$
are suppressed and 
 we simply write $\zeta(z)$ and $\sigma(z)$. The function $\sigma$ is odd, \texttt{quasi-periodic} \cite[page 470]{whittaker-2020a} and satisfies 
  the relations 
\begin{equation}
\sigma(z+ 2 \omega_{j})  =  - e^{2 \eta_{j} (z+ \omega_{j})} \sigma(z), \label{sigma-quasi} 
\end{equation}
\noindent
with $\eta_{j} = \zeta(\omega_{j})$.

Now use the identity \cite[page 473]{whittaker-2020a}
\begin{equation}
\label{P}
\wp(u) - \wp(v) = - \frac{\sigma(u+v) \sigma(u-v)}{\sigma^{2}(u) \sigma^{2}(v)},
\end{equation}
\noindent
to write 
\begin{eqnarray}
\wp(z) - \wp(\omega_{2})   =   -  \frac{\sigma(z+\omega_{2}) \sigma(z-\omega_{2})}{\sigma^{2}(z) \sigma^{2}(\omega_{2})}
 \label{P2}  
 =  \frac{\sigma^{2}(z+\omega_{2}) e^{-2 \eta_{2}z}}{\sigma^{2}(z) \sigma^{2}(\omega_{2})}, \nonumber 
\end{eqnarray}
\noindent
and convert  \eqref{int-j3} into 
\begin{multline*}
\J(c) = -\frac{1}{2} \int_{-\omega_{1}}^{\omega_{1}} \log(\sigma(z+ \omega_{2}))  \, dz 
+ \frac{1}{2} \int_{-\omega_{1}}^{\omega_{1}}  \log(\sigma(z)) \, dz +
 \omega_{1} \log(\sigma(\omega_{2})),
\end{multline*}
\noindent
since the integral arising from $e^{2 \eta_{2}z}$ vanishes. Define 
\begin{equation*}
L(\tau) := \int_{- \omega_{1}}^{\omega_{1}} \log \sigma(z + \tau) \, dz
\end{equation*}
%
%
\noindent
to write
\begin{equation}
\label{JL-1}
\J(c) =   \omega_{1} \log \sigma(\omega_{2}) -  \tfrac{1}{2} \left[ L(\omega_{2}) -  L(0) \right].
\end{equation}

The evaluation of $L(\tau)$ begins with
\begin{equation*}
\frac{d^{3}}{d \tau^{3}} L(\tau) = - \int \wp'(z+ \tau) \, dz = - \wp(\tau+\omega_{1}) + \wp(\tau - \omega_{1} )= 0
\end{equation*}
\noindent
since $2 \omega_{1}$ is a period of $\wp$. It follows that 
\begin{equation*}
L(\tau) = C_{0} \tau^{2} + C_{1}\tau + C_{2},
\end{equation*}
\noindent
for some constants $C_{0}, \, C_{1}, \, C_{2}$. 

The value $C_{0}$ comes from $J''(\tau)  = \zeta(\tau+\omega_{1}) - \zeta(\tau - \omega_{1})$ being constant. Now 
$L''(\tau) \equiv L''(0) = 2 \zeta(\omega_{1})$, since $\zeta$ is an odd function. This gives $C_{0} = \zeta(\omega_{1}) =  \eta_{1}$. 
%
The computation of $C_{1} = L'(0)$ uses the branch cut and gives 
\begin{eqnarray*}
C_{1} & = & \lim\limits_{\tau \rightarrow 0} \left( \log \sigma(\tau + \omega_{1}) - \log \sigma (\tau - \omega_{1}) \right)  \\
& = &  \lim\limits_{\tau \rightarrow 0} \left( \log \sigma(\tau + \omega_{1}) - \log \left[ -  \sigma (- \tau  +  \omega_{1}) \right]  \right)  \nonumber \\
& = &  - \pi i,  \nonumber 
\end{eqnarray*}
\noindent
since $\sigma$ is odd. Thus $L(\tau) = \eta_{1} \tau^{2} - \pi i \tau + C_{2}$. 
 Legendre's identity $\eta_{1} \omega_{2} - \eta_{2} \omega_{1} = \tfrac{1}{2} \pi i $ 
\cite[page 469]{whittaker-2020a} and \eqref{JL-1} yield
\begin{equation}
\label{formula-j10}
\J(c) = \omega_{1} \log \sigma(\omega_{2}) - \tfrac{1}{2} \eta_{2} \omega_{1}\omega_{2} +  \tfrac{1}{4} \pi i \omega_{2}.
\end{equation}

The next step requires an identity for the $\sigma$-function:

\begin{lemma}
\label{lemma-1}
The Weierstrass $\sigma$-function satisfies
\begin{equation}
\sigma^{2}(\omega_{1}+\omega_{2}) = e^{ 2 \eta_{2} \omega_{1}} \sigma^{2}(\omega_{1}) \sigma^{2}(\omega_{2}).
\label{V}
\end{equation}
\end{lemma}
\begin{proof}
Observe that \eqref{P2}
reveals
\begin{equation*} 
1  =  e_{1} - e_{3} 
 =  \wp(\omega_{1}) - \wp(\omega_{2}) 
 =  e^{- 2 \eta_{2} \omega_{1}} \frac{\sigma^{2}(\omega_{1}+\omega_{2})}{\sigma^{2}(\omega_{1}) \sigma^{2}(\omega_{2})}, 
\end{equation*}
\noindent
and this verifies the identity in \eqref{V}.
\end{proof}

The next result generates  an expression for $\log \sigma(\omega_{2})$.

\begin{lemma}
\label{lemma2}
The identity 
\begin{equation*}
c^{2} = \wp(\omega_{3}) - \wp(\omega_{2}) = \frac{e^{2 \eta_{2} \omega_{2}}}{\sigma^{4}(\omega_{2})},
\end{equation*}
\noindent
holds. Therefore $\log \sigma(\omega_{2}) = \tfrac{1}{2}  \eta_{2} \omega_{2} - \tfrac{1}{2}  \log c$.
\end{lemma}
\begin{proof}
\noindent
Use \eqref{P} combined with \eqref{sigma-quasi} to produce 
\begin{equation*}
c^{2} =e_{2} - e_{3} = \wp(\omega_{3}) - \wp(\omega_{2}) = - \frac{\sigma(\omega_{1} + 2 \omega_{2}) \sigma(\omega_{1})}{\sigma^{2}(\omega_{3}) \sigma^{2}(\omega_{2})} 
= \frac{\sigma^{2}(\omega_{1}) e^{2 \eta_{2} \omega_{3}}}{\sigma^{2}(\omega_{3}) \sigma^{2}(\omega_{2})},
\end{equation*}
\noindent
and the result follows from \eqref{V}. 
\end{proof}

Lemma \ref{lemma-1}, lemma \ref{lemma2},  equation  \eqref{formula-j10} and the expressions for the half-periods \cite[page 114]{mckean-1997a}
\begin{equation*}
\omega_{1} =  \sqrt{e_{1} - e_{3}} \, \K(c) \quad \textnormal{and} \quad \omega_{2} = i  \sqrt{e_{1}-e_{3}} \,  \K(\sqrt{1-c^{2}})
\end{equation*}
\noindent
and the fact that $e_{1} - e_{3} = 1$, readily show that 
\begin{equation}
\label{new-J1a}
\J(c) = - \frac{\pi }{4} \K(\sqrt{1- c^{2}}) -  \frac{1}{2} \K(c) \log c.
\end{equation}
\noindent
Replacing  \eqref{new-J1a} back in  \eqref{transf-1} completes the proof.
\end{proof}

\smallskip 

The computation of $I(a,b)$ presented above and Theorem \ref{thm-rep1} yield an expression for $F$ in terms of complete elliptic integrals.

\begin{corollary}
\label{coro-F1}
The function $F$  defined in \eqref{F-def} is given by
\begin{equation*}
F(x) = \frac{1}{\pi \sqrt{1+16x}} 
\left[ \ln \left( \frac{x}{1+16x} \right) \K \left( \sqrt{\frac{16x}{1+16x}} \right) + \pi \K \left( \frac{1}{\sqrt{1+16x}} \right) \right].
\end{equation*}
\end{corollary}

\section{A new  expression for the function $F$}
\label{sec-F}

The function $F$, defined in \eqref{F-def}, appeared in the first computation of $I(a,b)$ given in Theorem \ref{thm-rep1}.  This section
 presents a different approach to this function and 
establishes a connection with the integral $\J(c)$ defined in \eqref{J-def-0}. 

\begin{lemma}
\label{lem-ab}
Let $B_{n}$ be a sequence and define $A_{n}$ by $A_{n} = 4^{2n}(2n-1)B_{n}$. Then 
\begin{multline}
\label{int-mess1}
\sum_{n=0}^{\infty} (-1)^{n} B_{n} \binom{2n}{n}^{2} H_{n} = 
\frac{2}{\pi} \int_{0}^{1} \frac{\log(1-y^{2})}{\sqrt{1-y^{2}}} \left(  \sum_{n=0}^{\infty} A_{n} \binom{\tfrac{1}{2}}{n} y^{2n} \right) \, dy  \\
- 2 \log 2 \sum_{n=0}^{\infty} (-1)^{n} B_{n} \binom{2n}{n}^{2};
\end{multline}
\noindent
where $H_{n}$ are the harmonic numbers. 
\end{lemma}
\begin{proof}
Use the value 
\begin{equation*}
\int_{0}^{1} \frac{y^{2n} \log ( 1 - y^{2})}{\sqrt{1-y^{2}}}\, dy = - \frac{\pi}{2^{2n+1}} \binom{2n}{n} \left[ H_{n} + 2 \log 2 \right],
\end{equation*}
\noindent
to compute the integral in \eqref{int-mess1},  and \eqref{binom-form1} to simplify the result.
\end{proof}

The expression for $F$ is established next. 

\begin{theorem}
\label{thm-formF1}
The function $F$ defined in \eqref{F-def} is given by
\begin{equation}
\label{form-F10}
F(x) = - \frac{4}{\pi \sqrt{1+16 x}} 
\left[ \J \left( \frac{16x}{1+16x} \right) +\log 2 \, \K \left( \sqrt{\frac{16x}{1+16x}} \right) \right].
\end{equation}
\end{theorem}
\begin{proof}
Let $B_{n} = x^{n}$ in Lemma \ref{lem-ab} and use the classical series 
\begin{equation*}
\sum_{n=0}^{\infty} \binom{2n}{n} u^{n} =  (1 - 4u)^{-1/2} \quad \textnormal{and} \quad 
\sum_{n=0}^{\infty} \binom{2n}{n}^{2} u^{n} = \frac{2}{\pi} \K(\sqrt{16 u})
\end{equation*}
\noindent
to obtain 
\begin{equation}
\label{form-F1}
F(x) = - \frac{2}{\pi} \int_{0}^{1} \frac{\log(1- y^{2}) \, dy}{\sqrt{1-y^{2}} \sqrt{1 + 16 x y^{2}}} 
- \frac{4 \log 2}{\pi} \frac{1}{\sqrt{1+ 16 x }} \K \left( \sqrt{ \frac{16 x }{1+ 16 x }} \right),
\end{equation}
\noindent
using the transformation \eqref{K-imag1} for the complete elliptic integral.
The change of variables $u = 1-y^{2}$ converts \eqref{form-F1} into \eqref{form-F10} and produces  the result.
\end{proof}

\section{An automatic proof the last identity for  $F$}
\label{sec-automatic}

In this final section we show how the holonomic summation techniques  \cite{koutschan-2013a} can be employed 
to prove the closed-form evaluation of $F(x)$  stated in Corollary \ref{coro-F1}, restated here for the convenience of the reader:

\begin{theorem}
\label{thm-christoph}
The function $F$ in \eqref{F-def} is given by
\begin{eqnarray}
 \label{eq:idF}
\quad \quad F(x)  & = &  \frac{1}{\pi\sqrt{1+16x}}\left[\ln\biggl(\frac{x}{1+16x}\biggr)
    \K\biggl(\sqrt{\frac{16x}{1+16x}}\biggr)
    +\pi\K\biggl(\frac{1}{\sqrt{1+16x}}\biggr)\right]. 
\end{eqnarray}
\end{theorem}

\begin{proof}
Let $f_n(x)$ denote the expression inside the sum, i.e.,
\begin{equation*}
  f_n(x) = (-1)^n \binom{2n}{n}^2 H_nx^n.
\end{equation*}
Note that $f_n(x)$ is not hypergeometric in~$n$, due to the presence  of
harmonic numbers~$H_n$, and therefore the original Almkvist--Zeilberger
algorithm~\cite{almkvist-1990a} is not applicable. Instead it requires the corresponding 
 generalization to arbitrary holonomic functions, as implemented in the
\texttt{HolonomicFunctions}  package~\cite{koutschan-2010a}, which delivers the following
telescopic relation:
\begin{multline}
  x^2 (16 x+1)^2 f_n^{(4)}(x)
  +5 x (32 x+1) (16 x+1)f_n^{(3)}(x) \\
  +4 (1568 x^2+98 x+1) f_n''(x)
  {}+108 (32 x+1) f_n'(x)
  +144 f_n(x)  \\ = g_n(x) - g_{n+1}(x),
   \label{eq:tele}
\end{multline}
where $f_{n}^{(i)}$ denotes the $i^{th}$-derivative and 
\begin{equation*}
  g_n(x) = \frac{n}{x^2} \Bigl(\bigl((n-1) n^2+4 (2 n+1)^3 x\bigr) f_n(x)
  +n (n+1)^3 f_{n+1}(x)\Bigr).
\end{equation*}
The correctness of \eqref{eq:tele} can be established by routine calculations:
divide both sides by $\binom{2n+1}{n+1}^2(-x)^n$ and observe that it reduces to
\begin{equation*}
  (n+2) H_{n+2}-(2n+3) H_{n+1}+(n+1) H_n = 0,
\end{equation*}
which is indeed a valid relation for harmonic numbers. This can be established by writing the sums 
for the harmonic numbers and splitting in the manner
\begin{equation}
H_{n+2} = H_{n}  + \frac{1}{n+1} + \frac{1}{n+2} \textnormal{ and } 
H_{n+1} =  H_{n} + \frac{1}{n+1}.
\end{equation}

Summing the right-hand side of \eqref{eq:tele} over $n=0,\, 1, \dots$ gives,
for $|x|<\frac{1}{16}$,
\begin{equation*}
  g_0(x) - \lim_{n\to\infty}g_n(x) = 0,
\end{equation*}
while summing the left-hand side of \eqref{eq:tele} yields the desired fourth-order
differential equation for~$F(x)$:
\begin{multline}
  x^2 (16 x+1)^2 F^{(4)}(x)
  +5 x (32 x+1) (16 x+1)F^{(3)}(x) \\
  {}+4 (1568 x^2+98 x+1) F''(x)
  +108 (32 x+1) F'(x)
  +144 F(x) = 0.
   \label{eq:ODE}
\end{multline}

To derive a differential equation for the right-hand side of \eqref{eq:idF}, one
transforms the standard differential equation for the complete elliptic integral 
(see \cite[page 68]{mckean-1997a}):
\begin{equation*}
  x(x^2-1) \K''(x)+ (3 x^2-1) \K'(x) + x\,\K(x) = 0
\end{equation*}
into
\begin{equation*}
  x(16 x+1)^2 y''(x)+(16 x+1)^2 y'(x)-4 y(x) = 0,
\end{equation*}
which is satisfied by both 
\begin{equation*}
  y(x) = \K\biggl(\sqrt{\frac{16x}{1+16x}}\biggr)
  \quad\text{and}\quad
  y(x) = \K\biggl(\frac{1}{\sqrt{1+16x}}\biggr).
\end{equation*}
Combining it with the differential equation
\begin{equation*}
  x (16 x+1)^2 y''(x)+(48 x+1) (16 x+1) y'(x)+8 (24 x+1) y(x) = 0,
\end{equation*}
satisfied by
\begin{equation*}
  y(x) = \frac{1}{\sqrt{1+16x}}\ln\biggl(\frac{x}{1+16x}\biggr),
\end{equation*}
yields exactly the same differential equation as in \eqref{eq:ODE}. These kinds
of closure properties are executed algorithmically and automatically by the
\texttt{Annihilator} command in~\cite{koutschan-2010b}.

Since both sides of \eqref{eq:idF} satisfy the same fourth-order ODE, it
suffices to compare four initial values to establish equality. Using the
Taylor series
\begin{equation*}
  \K(x) = \pi\Bigl(\frac{1}{2} + \frac{1}{8}x^2 + \frac{9}{128}x^4 + \dots\Bigr)
\end{equation*}
one computes the series expansion 
\begin{equation*}
  -4x + 54x^2 - \frac{2200}{3} x^3 + \frac{30625}{3} x^4 + \dots
\end{equation*}
for the right-hand side of \eqref{eq:idF}. Truncating the sum on the left-hand
side produces exactly the same coefficients, thereby completing the proof.
\end{proof}

Replacing \eqref{eq:idF} in  Theorem \ref{thm-rep1} gives a proof of the original identity \eqref{entry-4.242.4}. The same procedure
 applies to $I(a,b)$ directly, at least in principle.
The result is a system of PDEs in~$a$ and~$b$, but it turns out that comparing
the initial values is more delicate.


\bibliography{/Users/vmh/Dropbox/AllRef/official1.bib}
\bibliographystyle{plain}
\end{document}